\documentclass{birkjour}
 \newtheorem{thm}{Theorem}[section]
 \newtheorem{cor}[thm]{Corollary}
 \newtheorem{lem}[thm]{Lemma}
 \newtheorem{prop}[thm]{Proposition}
 \theoremstyle{definition}
 \newtheorem{defn}[thm]{Definition}
 \theoremstyle{remark}
 \newtheorem{rem}[thm]{Remark}
 \newtheorem*{ex}{Example}
 \numberwithin{equation}{section}
\begin{document}
%----------------------------------------------------------------------------------------------------------------
% editorial commands: to be inserted by the editorial office
%\firstpage{1} \volume{228} \Copyrightyear{2004} \DOI{003-0001}
%\seriesextraline{This is the Concrete Title of this Book\br H.E. R and S.T.C. W, Eds.}
%\firstpage{1}
\title[On Combinatorial Properties of the DKAP]
{On Combinatorial Properties of the degenerate Krawtchouk Appell
polynomials}
%----------Authors-----------------------------------------------------------------------------------------------
%
\author{Mohamed Abdelkader}
\address{%
Department of Statistics and Operations Research, College of
Science, King Saud University, P.O. Box 2455, Riyadh 11451, Saudi
Arabia} \email{aabdelkader1@ksu.edu.sa}
\author{Mohamed Rhaima}
\address{%
Department of Statistics and Operations Research, College of
Sciences, King Saud University, P.O. Box 2455, Riyadh 11451, Saudi
Arabia} \email{mrhaima.c@ksu.edu.sa}
\thanks{}
\subjclass{Primary 33C45; Secondary 11C08, 94A11, 42C05, 94A11.}
\keywords{Bell polynomials, degenerate Pascal measure, degenerate
Krawtchouk Appell polynomials, scaling operator.}
\date{December 1, 2017}
\dedicatory{}
%%% ------------------------------------------------------------------------------------------------------------------------------
\begin{abstract}
The foremost aim of this study is to introduce and study several
combinatorial properties and highlight specific aspects of a new
class of polynomials sequences known as degenerate Krawtchouk Appell
polynomials associated with the degenerate Pascal measure. As
applications, the connection that exists between brand-new
polynomials, Stirling numbers, scaling operator, translation
operator and the orthogonal Bell polynomials has been investigated.
\end{abstract}
%%% -----------------------------------------------------------------------------------------------------------------------------
\maketitle
%%% -----------------------------------------------------------------------------------------------------------------------------
%\tableofcontents
%================================================================================================================================
\section{Introduction}
%================================================================================================================================
Over the past three decades, the study of various degenerate
versions of special polynomials has grown significantly in
prominence, primarily due to its demonstrated applications in many
different scientific and engineering domains as in
\cite{[a80],[d91],[jk19],[k97],[kaa24],[kaaa21],[kk18],[kk23],[kkkl20],
[kss97],[kss19],[kswy98],[ltdy24],[rare25]}.
In fact, its applications in many interesting arithmetical and
combinatorial results, it offers a number of useful tools for
resolving integral and differential equations, quantum probability
and statistics, as well as a number of other issues involving
particular mathematical physics functions, as well as their
extensions and generalizations. The first developments in this
context was pioneered by Carlitz \cite{[c79]}, who introduced
degenerate Stirling, Bernoulli and Euler polynomials by replacing
the exponential function $e^z$ with $(1+\lambda z)^{1/\lambda}$ in
their generating functions. This approach recovers the classical
polynomials in the limit as $\lambda \to 0$. Building upon these
developments, Haroon and Khan \cite{[hk18]} investigated degenerate
Hermite-Bernoulli polynomials using $p$-adic fermionic integrals,
leading to new identities that establish connections with Daehee
polynomials in a unified and generalized framework. Duran and
Acikgoz \cite{[da20],[da20b]} introduced the generalized degenerate
Gould-Hopper-based Stirling polynomials of the second kind and the
degenerate Poisson distribution as well as is related to degenerate
Bell polynomials. In a related direction, Further contributions
include the work of Araci \cite{[a21]} and \cite{[a21b]}, who
studied a new class of generating functions for type $2$ Bernoulli
polynomials and a new family of Daehee polynomials called degenerate
$q$-Daehee polynomials, deriving several identities involving type
$2$ Euler polynomials and Stirling numbers of the second kind.
Frontczak and Tomovski \cite{[ft22]} provided probabilistic
interpretations for summation formulas involving generalized
Bernoulli and Euler-Genocchi polynomials of order $(r, m)$,
enriching the analytical framework of these polynomial families.
Recent work by Kim and Kim \cite{[kk23]} introduced probabilistic
variants of degenerate Stirling numbers of the second kind and
degenerate Bell polynomials, grounded in random variables that
satisfy certain moment conditions. These advancements continue to
deepen the interplay between combinatorial structures, probabilistic
methods, and degenerate special functions, revealing new pathways
for analytical exploration and interdisciplinary applications.

In recent years, there has been a great deal of interest in the
study of composite Poisson processes and their application to
various problems. The Pascal measure is an important example of
compound Poisson which is useful and popular discrete probability
distribution in probability theory and statistics. Moreover the
Krawtchouk polynomials, originally introduced by Mikhail Kravchuk
(Krawtchouk), represent a foundational class of orthogonal
polynomials with deep connections to combinatorics and probability
theory. A comprehensive overview of Kravchuk's contributions and the
historical development of these polynomials up to $2004$ is
available in \cite{[vk04]}. These polynomials possess intrinsic
mathematical richness and have found widespread applications across
diverse scientific fields, including biology, complex networks,
blockchain technology, quantum probability, coding theory,
electromagnetic wave propagation, and image processing (see
\cite{[mtv23]} about the notion of its Darboux transformations and
\cite{[yp03]} for applications in image analysis). Recall that the
probability mass function of Pascal random variable $X$ is given by
$$
Pr([X=n]):=\mathcal{NB}_{p,r}(\{n\})=p^{r}\binom{-r}{n}(-q)^n, \quad
q :=1-p, \; r>0,
$$
where for $\upsilon\in \mathbb{R}$ the generalized binomial
coefficient is defined by
$$
\binom{\upsilon}{n}=\frac{\upsilon(\upsilon-1)\ldots(\upsilon-n+1)}{n!}.
$$
For $s\in \mathbb{R}$ such that $|qe^{s}|<1$, one can see that the
Laplace transform of $\mathcal{N}\mathcal{B}_{p, r}$ denoted by
$\mathcal{L}_{\mathcal{N}\mathcal{B}_{p, r}}$ is given by
\begin{equation}\label{Laplace-Pascal}
\mathcal{L}_{\mathcal{N}\mathcal{B}_{p, r}}(s)
=\displaystyle{p^{r}\sum^{\infty}_{k=0}\binom{-r}{k}(-qe^{s})^{k}=\Big(\frac{p}{1-qe^{s}}\Big)^{r}}.
\end{equation}
It is clear that the generating function of the so-called classical
Krawtchouk polynomials $\mathcal{K}_{n, r}(x)$ is given by
\begin{equation}\label{Generating-Pascal}
\Psi(t,
x)=\displaystyle{(1+t)^{x}(1+qt)^{-x-r}=\sum^{\infty}_{n=0}\frac{t^{n}}{n!}\mathcal{K}_{n,
r}(x)}.
\end{equation}
Asai, Kubo and Kuo in \cite{[acc03]} prove that the orthogonality of
polynomials generated by $\Psi(t, x)$ is guaranteed under the
condition that $\mathbb{E}_{\mathcal{NB}_{p, r}}(\Psi(s, x)\Psi(t,
x))$ is a function of $st$. This approach cannot be applicable in
the degenerate version of Pascal case. So to solve this problem, the
so-called normalized-Laplace exponential is introduced to produce a
new class of polynomials which are called the degenerate Krawtchouk
Appell polynomials.

The structure of the paper is as follows. In section $2$, we define
the degenerate Pascal measure $\mathcal{NB}^{(\lambda)}_{p, r}$ and
we prove that this measure is a combination of the classical Pascal
measure and a specific probability measure $\nu_{\lambda, \beta}$ on
$\mathbb{R}_{+}$. Next we show that the $n$-th moments of
$\mathcal{NB}^{(\lambda)}_{p, r}$ is given explicitly in terms of
Stirling numbers of the second kind. In section 3, we introduce and
highlight key properties of a novel class of Appell polynomials
associated with the degenerate Pascal measure and as application we
obtain explicitly the chaos expansion of the scaling operator in
terms of these polynomials.
%==================================================================
\section{Degenerate Pascal Measure}
%==================================================================
Our first goal is to give the degenerate version of the Pascal
measure. So to define degenerate extensions of classical functions
and polynomials, we introduce the degenerate exponential function.
For any $\lambda\in \mathbb{R}$ and $z\in \mathbb{C}$, the
degenerate exponential function is defined as
\begin{equation}\label{Degenerate-exp}
e^{\beta}_{\lambda}(z)=(1+\lambda
z)^{\frac{\beta}{\lambda}}=\displaystyle{\sum^{\infty}_{k=0}(\beta)_{k,
\lambda}\frac{z^{k}}{k!}},
\end{equation}
where $(\beta)_{k,\lambda}$ is the degenerate falling factorial
defined by
\begin{equation}\label{Degenerate-falling}
(\beta)_{k,\lambda}:=
\begin{cases}
1, & \text{if } k=0, \\
\beta(\beta- \lambda)(\beta- 2\lambda)\cdots (\beta-(k-1)\lambda), &
\text{if}\, k \in \mathbb{N}^{\ast},
\end{cases}
\end{equation}
with the convention $(0)_{k,\lambda}:=1$.
\begin{defn}
For $\lambda< 0$ and $\beta>0$, the degenerate Pascal measure
$\mathcal{N}\mathcal{B}^{(\lambda)}_{p, r}$ is defined on the
$\sigma$-algebra of all subsets of $\mathbb{N}$ denoted by
$\mathcal{P}(\mathbb{N})$ by
\begin{equation}\label{Degenerate_Pascal-measure}
\mathcal{N}\mathcal{B}^{(\lambda)}_{p, r}(\Lambda)=
\displaystyle{\sum^{\infty}_{k=0}\frac{1}{k!}(e^{\beta}_{\lambda})^{(k)}
\Big[r\log\Big(\frac{p}{1-qz}\Big)\Big]_{z=0}\delta_{k}(\Lambda)},\quad
\Lambda\in \mathcal{P}(\mathbb{N}),
\end{equation}
where $\delta_{k}$ is the Dirac measure at $k\in \mathbb{N}$ and
$(e^{\beta}_{\lambda})^{(k)}(z):=\frac{d^{k}}{dz^{k}}e^{\beta}_{\lambda}(z)$
is the $k$-th derivative of $e^{\beta}_{\lambda}$. For $z\in
\mathbb{C}$ such that $|qe^{z}|<1$, the Laplace transform of the
measure $\mathcal{N}\mathcal{B}^{(\lambda)}_{p, r}$ is given by
\begin{equation}\label{laplace-tr-Degen-Pascal}
\mathcal{L}_{\mathcal{N}\mathcal{B}^{(\lambda)}_{p, r}}(z)
=\displaystyle{\sum^{\infty}_{k=0}
\frac{(qe^{z})^{k}}{k!}(e^{\beta}_{\lambda})^{(k)}
\big(r\log(p)\big)}
=\displaystyle{e^{\beta}_{\lambda}\Big\{r\log\big(\frac{p}{1-qe^{z}}\big)\Big\}}.
\end{equation}
\end{defn}
\begin{prop}
The degenerate Pascal measure
$\mathcal{N}\mathcal{B}^{(\lambda)}_{p, r}$ is related to the Pascal
measure $\mathcal{N}\mathcal{B}_{p, r}$ under the following
relation:
\begin{equation}\label{formula1}
\mathcal{N}\mathcal{B}^{(\lambda)}_{p,
r}=\displaystyle{\int^{\infty}_{0}\mathcal{N}\mathcal{B}_{p,
rs}\nu_{\lambda, \beta}(ds)},
\end{equation}
where
\begin{equation}\label{G-func}
\displaystyle{\nu_{\lambda,
\beta}(ds)=\frac{s^{-\frac{\beta}{\lambda}-1}e^{\frac{s}{\lambda}}}{\Gamma(-\frac{\beta}{\lambda})
(-\lambda)^{-\frac{\beta}{\lambda}}}ds:=G_{\lambda, \beta}(s)ds}.
\end{equation}
\end{prop}
\begin{proof}
Firstly we need to determine a function $G_{\lambda, \beta}$ such
that
\begin{equation}\label{Laplace-Transform-G}
\displaystyle{\int_{0}^{\infty} e^{-sx} G_{\lambda, \beta}(x)dx=
e^{\beta}_{\lambda} (-s)=(1-\lambda
s)^{\frac{\beta}{\lambda}}=\lambda^{\frac{\beta}{\lambda}}\Big(\frac{1}{\lambda}-s\Big)^{\frac{\beta}{\lambda}}}.
\end{equation}
Then by using the following Laplace transform
\begin{equation*}
(-1)^k\mathcal{L}(x^{k-1}e^{ax})(t)=\frac{\Gamma (k)}{(a-t)^{k}},
\quad k >0, \quad \mathcal{R}(a-t)> 0,
\end{equation*}
and by choosing $\displaystyle{a=\frac{1}{\lambda},
k=-\frac{\beta}{\lambda}}$, we get
\begin{equation*}
\lambda^{\frac{\beta}{\lambda}}\Big(\frac{1}{\lambda}-s
\Big)^{\frac{\beta}{\lambda}}=\mathcal{L}\Big(
\frac{x^{-\frac{\beta}{\lambda}-1}e^{\frac{x}{\lambda}}}{\Gamma(-\frac{\beta}{\lambda})
(-\lambda)^{-\frac{\beta}{\lambda}}}\Big)(s).
\end{equation*}
Hence the function $G_{\lambda, \beta}$ is given explicitly by
\begin{equation}\label{G-func}
G_{\lambda, \beta}(x) =
\frac{x^{-\frac{\beta}{\lambda}-1}e^{\frac{x}{\lambda}}}{\Gamma(-\frac{\beta}{\lambda})
(-\lambda)^{-\frac{\beta}{\lambda}}}.
\end{equation}
Note that $G_{\lambda, \beta}$ is the probability density function
of the Gamma distribution with shape parameter
$-\frac{\beta}{\lambda}$ and scale parameter $-\frac{1}{\lambda}$.
Denote by
$$\mu:=\displaystyle{\int^{\infty}_{0}\mathcal{N}\mathcal{B}_{p, ry}G_{\lambda, \beta}(y)dy}.$$
Then by using equations \eqref{laplace-tr-Degen-Pascal},
\eqref{Laplace-Transform-G} and Fubini's theorem, we get
$$\begin{array}{lll}
\displaystyle{\int^{\infty}_{0}e^{zx}\mu(dx)}&=&
\displaystyle{\int^{\infty}_{0}e^{zx}\int^{\infty}_{0}\mathcal{N}\mathcal{B}_{p, ry}(dx)G_{\lambda, \beta}(y)dy}\\
&=&\displaystyle{\int^{\infty}_{0}\Big(\int^{\infty}_{0}e^{zx}\mathcal{N}\mathcal{B}_{p, ry}(dx)\Big)G_{\lambda, \beta}(y)dy}\\
&=&\displaystyle{\int^{\infty}_{0}\exp\Big[ry\log\big(\frac{p}{1-qe^{z}}\big)\Big]G_{\lambda, \beta}(y)dy}\\
&=&\displaystyle{e^{\beta}_{\lambda}\Big\{r\log\Big(\frac{p}{1-qe^{z}}\Big)\Big\}}\\
&=&\mathcal{L}_{\mathcal{N}\mathcal{B}^{(\lambda)}_{p, r}}(z).
\end{array}$$
This gives the statement.
\end{proof}
\begin{prop}
The $m$-th moment of
$\mathcal{N}\mathcal{B}^{(\lambda)}_{p, r}$ is given explicitly by\\

$M_{\lambda}(m):=\displaystyle{\int_{\mathbb{R}}x^{m}\mathcal{N}\mathcal{B}^{(\lambda)}_{p,
r}(dx)}$
\begin{equation}\label{m-th-Moments}
=\displaystyle{(1+\lambda r\log p)^{\frac{\beta}{\lambda}}
\sum^{n}_{k=1}S(m, k)\Big(\frac{q}{1+\lambda r\log
p}\Big)^{k}(\beta)_{k, \lambda}e^{\beta-\lambda
k}_{\lambda}\Big(\frac{q}{1+\lambda r\log p}\Big)},
\end{equation}
where $S(m, k)$ is the Stirling numbers of the second kind defined
via its generating function by
\begin{equation}\label{Stir-2}
\displaystyle{(e^{z}-1)^{k}=k!\sum^{\infty}_{m=k}S(m,
k)\frac{z^{m}}{m!}}.
\end{equation}
\end{prop}
\begin{proof}
Firstly we need to determine the $n$-th derivative of the degenerate
exponential function. If we put $\varphi(t)=e^{\beta}_{\lambda}(t)$
and differentiating \eqref{Degenerate-exp}, then we obtain
$$\begin{array}{lll}
\varphi'(t)&=&\beta(1+\lambda t)^{\frac{\beta}{\lambda}-1}\\
\varphi''(t)&=&\beta(\beta-\lambda)(1+\lambda t)^{\frac{\beta}{\lambda}-2}\\
\varphi^{3}(t)&=&\beta(\beta-\lambda)(\beta-2\lambda)(1+\lambda t)^{\frac{\beta}{\lambda}-3}\\
\varphi^{(4)}(t)&=&\beta(\beta-\lambda)(\beta-2\lambda)(\beta-3\lambda)(1+\lambda
t)^{\frac{\beta}{\lambda}-4}.
\end{array}$$
Thus by continuing the process $n$ number of times, gives
$$\varphi^{(n)}(t)=\displaystyle{(1+\lambda t)^{\frac{\beta}{\lambda}-n}\prod^{n-1}_{k=0}(\beta-k\lambda)
=(\beta)_{n, \lambda}e^{\beta-\lambda n}_{\lambda}(t)},$$ and we get
\begin{equation}\label{Degen-Pascal-formula2}
\mathcal{N}\mathcal{B}^{(\lambda)}_{p,
r}(\{n\})=\displaystyle{\frac{q^{n}}{n!}\varphi^{(n)}(r\log(p))
=\frac{q^{n}}{n!}(\beta)_{n, \lambda}e^{\beta-\lambda
n}_{\lambda}(r\log p)}.
\end{equation}
Hence by using \eqref{Degenerate-exp} and
\eqref{Degen-Pascal-formula2} we get
\begin{eqnarray}\label{moments-F1}
\displaystyle{\int_{\mathbb{R}}x^{m}\mathcal{N}\mathcal{B}^{(\lambda)}_{p,
r}(dx)}&=&\displaystyle{\sum^{\infty}_{n=0}n^{m}\mathcal{N}\mathcal{B}^{(\lambda)}_{p, r}(\{n\})}\nonumber\\
&=&\displaystyle{\sum^{\infty}_{n=0}n^{m}\frac{q^{n}}{n!}(\beta)_{n,
\lambda}e^{\beta-\lambda n}_{\lambda}(r\log p)}\nonumber\\
&=&\displaystyle{\sum^{\infty}_{n=0}n^{m}\frac{q^{n}}{n!}(\beta)_{n,
\lambda}(1+\lambda r\log p)^{\frac{\beta}{\lambda}-n}}\nonumber\\
&=&\displaystyle{(1+\lambda r\log
p)^{\frac{\beta}{\lambda}}\sum^{\infty}_{n=0}n^{m}(\beta)_{n,
\lambda}\frac{(\frac{q}{1+\lambda r\log p})^{n}}{n!}}.
\end{eqnarray}
To handle the $n^{m}$ term, we can use the operator
$\displaystyle{x\frac{d}{dx}}$ which acts on $x^{n}$ as:
$$\displaystyle{x\frac{d}{dx}x^{n}=nx^{n}}.$$
Applying this operator $m$ times gives
$$\displaystyle{\Big(x\frac{d}{dx}\Big)^{m}x^{n}=n^{m}x^{n}}.$$
Thus we obtain
\begin{equation}\label{moments-F2}
\displaystyle{\sum^{\infty}_{n=0}n^{m}(\beta)_{n,
\lambda}\frac{x^{n}}{n!}=\Big(x\frac{d}{dx}\Big)^{m}(1+\lambda
x)^{\frac{\beta}{\lambda}}}.
\end{equation}
Moreover, the operator $\displaystyle{\Big(x\frac{d}{dx}\Big)^{m}}$
can be expressed in terms of $S(m, k)$ as follows:
$$\displaystyle{\Big(x\frac{d}{dx}\Big)^{m}=\sum^{m}_{k=1}S(m, k)x^{k}\frac{d^{k}}{dx^{k}}}.$$
Applying this to $\displaystyle{(1+\lambda
x)^{\frac{\beta}{\lambda}}}$ to obtain
\begin{eqnarray}\label{moments-F3}
\displaystyle{\Big(x\frac{d}{dx}\Big)^{m}(1+\lambda
x)^{\frac{\beta}{\lambda}}}&=&\displaystyle{\sum^{m}_{k=1}S(m,
k)x^{k}\frac{d^{k}}{dx^{k}}(1+\lambda x)^{\frac{\beta}{\lambda}}}\nonumber\\
&=&\displaystyle{\sum^{m}_{k=1}S(m, k)x^{k}(\beta)_{k,
\lambda}e^{\beta-\lambda k}_{\lambda}(x)}.
\end{eqnarray}
Therefore \eqref{moments-F1}, \eqref{moments-F2} and
\eqref{moments-F3} yields
$$M_{\lambda}(m)=\displaystyle{(1+\lambda r\log(p))^{\frac{\beta}{\lambda}}
\sum^{m}_{k=1}S(m, k)\Big(\frac{q}{1+\lambda r\log
p}\Big)^{k}(\beta)_{k, \lambda}e^{\beta-\lambda
k}_{\lambda}\Big(\frac{q}{1+\lambda r\log p}\Big)}.$$
\end{proof}
%==========================================================================================================================
\section{Degenerate Krawtchouk Appell Polynomials}
%===========================================================================================================================
Our first main objective is to introduce via multiplicative
generating function and to study a new system of Appell polynomials
associated with the degenerate Pascal measure that do not form an
orthogonal family.
\begin{defn}
Consider the following entire function $\vartheta$ defined on the
open unit disk by
$$\vartheta(z):=\log\Big(\frac{1+z}{1+qz}\Big), \quad |z|<1.$$
For $z\in \mathbb{C}$ such that $|qe^{z}|<1$, the normalized
Laplace-exponential with respect to
$\mathcal{N}\mathcal{B}^{(\lambda)}_{p, r}$ is defined by
\begin{equation}\label{wick-exp}
\mathcal{N}\!\exp^{\beta}_{\lambda}(z,
x):=\displaystyle{\frac{e^{xz}}{\mathcal{L}_{\mathcal{N}\mathcal{B}^{(\lambda)}_{p,
r}}(z)}
=\frac{e^{xz}}{e^{\beta}_{\lambda}\big[r\log(\frac{p}{1-qe^{z}})\big]}}.
\end{equation}
Therefore, there exist a neighborhood $\mathcal{U}$ of $0\in
\mathbb{C}$ given by
$$\mathcal{U}=\big\{z\in \mathbb{C}, |qe^{z}|<1\big\}$$
such that $\mathcal{N}\!\exp^{\beta}_{\lambda}(\vartheta(\cdot), x)$
is given in terms of the so-called \textit{degenerate Krawtchouk
Appell polynomials} $\mathcal{K}^{(\beta)}_{n, \lambda}$ by
\begin{equation}\label{Gen-functio-form1}
\mathcal{N}\!\exp^{\beta}_{\lambda}(\vartheta(z),
x)=\displaystyle{\sum^{\infty}_{n=0}\frac{z^{n}}{n!}
\mathcal{K}^{(\beta)}_{n, \lambda}(x), \quad\forall z\in
\mathcal{U}},
\end{equation}
with
$$\mathcal{K}^{(\beta)}_{n, \lambda}(x)=\displaystyle{\frac{d^{n}}{dz^{n}}
\Big[\mathcal{N}\!\exp^{\beta}_{\lambda}(\vartheta(z),
x)\Big]_{z=0}}.$$
\end{defn}
Now our goals are to find an explicit form of these polynomials. One
can see that $\mathcal{N}\!\exp^{\beta}_{\lambda}(\vartheta(\cdot),
x)$ can be considered as a generating function of new polynomial
$\mathcal{K}^{(\beta)}_{n, \lambda}(x)$ denoted by
$\Psi_{\lambda}(t, x)$. That is
\begin{equation}\label{Generatin-form2}
\Psi_{\lambda}(t,
x)=\displaystyle{\frac{e^{x\vartheta(t)}}{\mathcal{L}_{\mathcal{N}\mathcal{B}^{(\lambda)}_{p,
r}}(\vartheta(t))}=\mathcal{N}\!\exp^{\beta}_{\lambda}(\vartheta(t),
x)=\sum^{\infty}_{n=0}\frac{z^{n}}{n!}\mathcal{K}^{(\beta)}_{n,
\lambda}(x)}.
\end{equation}
\begin{rem}
The polynomials $\mathcal{K}^{(\beta)}_{n, \lambda}(x)$ do not form
a system of orthogonal polynomials because
$\mathbb{E}_{\mathcal{NB}^{(\lambda)}_{p, r}}(\Psi_{\lambda}(s,
x)\Psi_{\lambda}(t, x))$ is not a function of $st$. In fact, by
using (\ref{Generatin-form3}) then for $s, t>0$, we have
\begin{eqnarray}\label{Eq1}
\mathbb{E}_{\mathcal{NB}^{(\lambda)}_{p, r}}(\Psi_{\lambda}(s,
x)\Psi_{\lambda}(t, x))
&=&\displaystyle{\sum^{\infty}_{k=0}\frac{e^{k(\vartheta(t)+\vartheta(s))}}{\mathcal{L}_{\mathcal{N}\mathcal{B}^{(\lambda)}_{p,
r}}(\vartheta(t))
\mathcal{L}_{\mathcal{N}\mathcal{B}^{(\lambda)}_{p, r}}(\vartheta(s))}\mathcal{NB}^{(\lambda)}_{p, r}(\{k\})}\nonumber\\
&=&\displaystyle{\frac{\mathcal{L}_{\mathcal{N}\mathcal{B}^{(\lambda)}_{p,
r}}(\vartheta(t)+\vartheta(s))}
{\mathcal{L}_{\mathcal{N}\mathcal{B}^{(\alpha)}_{p, r}}(\vartheta(t))
\mathcal{L}_{\mathcal{N}\mathcal{B}^{(\lambda)}_{p, r}}(\vartheta(s))}}\nonumber\\
&=&\displaystyle{\frac{e^{\beta}_{\lambda}\Big[r\log\big(\frac{p(1+qt)(1+qs)}{(1+qt)(1+qs)
-q(1+t)(1+s)}\big)\Big]}{e^{\beta}_{\lambda}(r\log(1+qt))e^{\beta}_{\lambda}(r\log(1+qs))}}.
\end{eqnarray}
\end{rem}
\begin{thm}
The degenerate Krawtchouk Appell polynomials posses the following
form
\begin{equation}\label{Fractional-Krawtchouk-poly}
\mathcal{K}^{(\beta)}_{n,
\lambda}(x)=\displaystyle{\sum^{n}_{k=0}\frac{n!}{(n-k)!}r^{n-k}c_{n-k}(r)\varepsilon_{k}(x)},
\end{equation}
where
\begin{equation}\label{def-omega}
\varepsilon_{k}(x)=\displaystyle{\sum^{k}_{j=0}(-1)^{j}
\binom{x}{k-j}\binom{x+j-1}{k-j}q^{j}},
\end{equation}
and the the coefficient sequence $(c_{n}(r))_{n}$ fulfill the
following recurrence relation:
$$\begin{cases}
c_{0}(r)=1\\
c_{n}(r)=-\displaystyle{\sum^{n}_{i=1}
\binom{n}{i}r^{-i}(e^{\beta}_{\lambda}\circ
\xi_{q})^{(i)}(0)c_{n-i}(r) }\quad\text{with}\quad
\xi_{q}(t)=r\log(1+qt).
\end{cases}
$$
\end{thm}
\begin{proof}
By using equation (\ref{Generatin-form2}) we have
\begin{equation}\label{Generatin-form3}
\Psi_{\lambda}(t,
x)=\displaystyle{\frac{\omega(t)^{x}}{e^{\beta}_{\lambda}(r\log(1+qt))}=
\sum^{\infty}_{n=0}\frac{t^{n}}{n!}\mathcal{K}^{(\beta)}_{n,
\lambda}(x)},
\end{equation}
where $\omega(t)=\frac{1+t}{1+qt}$. Note that for $x=0$, we get
\begin{equation}\label{def-phi-t}
\Omega(t):=\Psi_{\lambda}(t,
0)=\displaystyle{\frac{1}{e^{\beta}_{\lambda}(r\log(1+qt))}=
\sum^{\infty}_{n=0}\frac{t^{n}}{n!}\mathcal{K}^{(\beta)}_{n,
\lambda}(0)}
\end{equation}
Now we put
\begin{equation}\label{Gen-fun-at0}
\Omega(t)=\displaystyle{\sum^{\infty}_{k=0}c_{k}(r)\frac{(rt)^{k}}{k!}}.
\end{equation}
The binomial series expansion and cauchy products rules yields
\begin{eqnarray}
\omega(t)^{x}&=&
\displaystyle{\sum^{\infty}_{n=0}\Big[\sum^{n}_{j=0}
(-1)^{j}\binom{x}{n-j}\binom{x+j-1}{j}q^{k}\Big]t^{n}}\nonumber\\
&=&\displaystyle{\sum^{\infty}_{k=0}\varepsilon_{k}(x)t^{k}},
\end{eqnarray}
where $\varepsilon_{n}$ is given as in equation (\ref{def-omega}).
Thus by using equation (\ref{Gen-fun-at0}) we get
\begin{eqnarray}\label{Generatin-form4}
\Psi_{\lambda}(t, x)&=&\displaystyle{\omega(t)^{x}\sum^{\infty}_{k=0}c_{k}(r)\frac{(rt)^{k}}{k!}}\nonumber\\
&=&\displaystyle{\sum^{\infty}_{k=0}\varepsilon_{k}(x)t^{k}\sum^{\infty}_{k=0}c_{k}(r)\frac{(rt)^{k}}{k!}}\nonumber\\
&=&\displaystyle{\sum^{\infty}_{n=0}t^{n}\Big(\sum^{n}_{k=0}\frac{r^{n-k}}{(n-k)!}\varepsilon_{k}(x)c_{n-k}(r)\Big)}.
\end{eqnarray}
Thus (\ref{Generatin-form3}) and (\ref{Generatin-form4}) yields
$$\mathcal{K}^{(\beta)}_{n, \lambda}(x)=\displaystyle{\sum^{n}_{k=0}\frac{n!}{(n-k)!}r^{n-k}c_{n-k}(r)\varepsilon_{k}(x)}.$$
Now our gaol to find coefficients $c_{n}(r)$. So by using
\eqref{def-phi-t} and \eqref{Gen-fun-at0} we have
\begin{equation}\label{condition-coeff}
\displaystyle{e^{\beta}_{\lambda}\circ
\xi_{q}(t)\cdot\sum^{n}_{k=0}c_{k}(r)\frac{(rt)^{k}}{k!}=1}.
\end{equation}
That's to say
$$\displaystyle{\Big(\sum^{\infty}_{k=0}\frac{(e^{\beta}_{\lambda}\circ \xi_{q})^{(k)}(0)}{k!}t^{k}\Big)
\Big(\sum^{n}_{k=0}c_{k}(r)\frac{(rt)^{k}}{k!}\Big)=1}.$$ Moreover
we have
$$\displaystyle{\Big(\sum^{\infty}_{k=0}\frac{(e^{\beta}_{\lambda}\circ \xi_{q})^{(k)}(0)}{k!}t^{k}\Big)
\Big(\sum^{n}_{k=0}c_{k}(r)\frac{(rt)^{k}}{k!}\Big)=\sum^{\infty}_{n=0}\gamma_{n}(r)\frac{t^{n}}{n!}},$$
with
\begin{eqnarray}\label{def-Cnr}
\gamma_{n}(r)&=&\displaystyle{\sum^{n}_{i=0}\frac{(e^{\beta}_{\lambda}\circ \xi_{q})^{(i)}(0)}{i!}\frac{n!c_{n-i}(r)r^{n-i}}{(n-i)!}}\nonumber\\
&=&\displaystyle{\sum^{n}_{i=0}
\binom{n}{i}r^{n-i}(e^{\beta}_{\lambda}\circ
\xi_{q})^{(i)}(0)c_{n-i}(r)}.
\end{eqnarray}
From equations \eqref{condition-coeff} and \eqref{def-Cnr} we
conclude that
$$\gamma_{0}(r)=1,\quad \gamma_{n}(r)=0\quad \forall n\geq1.$$
Thus we get
$$c_{0}(r)=\gamma_{0}(r)=1,$$
and
$$\displaystyle{\sum^{n}_{i=0}
\binom{n}{i}r^{n-i}(e^{\beta}_{\lambda}\circ
\xi_{q})^{(i)}(0)c_{n-i}(r)=0\quad\forall n\geq1},$$ which gives the
following recurrence formula
$$c_{n}(r)=-\displaystyle{\sum^{n}_{i=1}
\binom{n}{i}r^{-i}(e^{\beta}_{\lambda}\circ
\xi_{q})^{(i)}(0)c_{n-i}(r)}.$$
\end{proof}
\begin{ex}
A simple calculus one can see that the first and the second
coefficients are given by
$$
c_{1}(r)=\displaystyle{-\frac{1}{r}(e^{\beta}_{\lambda}\circ
\xi_{q})'(0)c_{0}(r)=-q\beta},
$$
$$\begin{array}{lll}
c_{2}(r)&=&\displaystyle{-\frac{2}{r}r^{-1}(e^{\beta}_{\lambda}\circ \xi_{q})'(0)c_{1}(r)-r^{-2}(e^{\beta}_{\lambda}\circ \xi_{q})''(0)c_{0}(r)}\\
&=&\displaystyle{\frac{2q^{2}\beta^{2}}{r}-\frac{\beta
q^{2}}{r}+(1-\frac{\beta}{\lambda})\lambda\beta q^{2}}.
\end{array}$$
As a consequence we can present the first three degenerate
Krawtchouk Appell polynomials:
$$\begin{array}{lll}
\mathcal{K}^{(\beta)}_{0, \lambda}(x)&=&c_{0}(r)=1\\
\mathcal{K}^{(\beta)}_{1, \lambda}(x)&=&\displaystyle{x-\beta\frac{rq}{p}}\\
\mathcal{K}^{(\beta)}_{2, \lambda}(x)&=&\displaystyle{x^{2}-\Big(\frac{2rq\beta}{p}+\frac{1-q^{2}}{p^{2}}\Big)x}\\
&+&\displaystyle{\frac{r^{2}}{p^{2}}\Big(\frac{2q^{2}\beta^{2}}{r}-\frac{\beta
q^{2}}{r}+(1-\frac{\beta}{\lambda})\lambda\beta q^{2}\Big)}
\end{array}$$
At $\beta= 1$ and $\lambda\rightarrow 0$ these polynomials become
the well known Krawtchouk polynomials.
\end{ex}
Now, consider the Taylor expansion of normalized Laplace-exponential
at $z$ given in terms of Appell polynomials $P^{(\beta)}_{n,
\lambda}, n \in \mathbb{N}$ by
\begin{equation}\label{Exp-alpha-1}
\mathcal{N}\!\exp^{\beta}_{\lambda}(z,
x)=\displaystyle{\frac{e^{xz}}{\mathcal{L}_{\mathcal{N}\mathcal{B}^{(\lambda)}_{p,
r}}(z)}= \sum^{\infty}_{n=0}\frac{z^{n}}{n!}P^{(\beta)}_{n,
\lambda}(x)},\quad z\in \mathcal{U}.
\end{equation}
\begin{thm}
The polynomials $\mathcal{K}^{(\beta)}_{n, \lambda}$ satisfy the
following useful combinatorial properties:
\begin{itemize}
\item [\textbf{(P1)}]
$\mathcal{K}^{(\beta)}_{n, \lambda}(x)=\displaystyle{\sum^{n}_{m=0}
\frac{\varpi(m, n)}{m!}P^{(\beta)}_{m, \lambda}(x)}$, where
$\varpi(0, 0):=1$ and
\begin{equation}\label{def-lambda-mn}
\varpi(m,
n):=\displaystyle{\sum_{i_{1}+\ldots+i_{m}=n}(-1)^{n+m}n!\prod^{m}_{j=1}\Big(\frac{1-q^{i_{j}}}{i_{j}}\Big)}.
\end{equation}
\item [\textbf{(P2)}]
$x^{n}=\displaystyle{\sum^{n}_{k=0}\sum^{k}_{m=0}\binom{n}{k}
\frac{\varrho(m, k)}{m!}\mathcal{K}^{(\beta)}_{m,
\lambda}(x)M_{\lambda}(n-k)}$, where $M_{\lambda}(k)$ is the $k$-th
moment given by equation \eqref{m-th-Moments} and
$$\begin{cases}
\varrho(0, 0):=1\\
\varrho(m,
k):=\displaystyle{\sum_{i_{1}+\cdots+i_{m}=k}\frac{k!}{i_{1}!\ldots
i_{m}!}\prod^{m}_{j=1}\frac{d^{i_{j}}}{dz^{i_{j}}}\Big[\frac{e^{z}-1}{1-qe^{z}}\Big]_{z=0}
}.
\end{cases}
$$
\item [\textbf{(P3)}]
$\mathcal{K}^{(\beta)}_{n,
\lambda}(x+y)=\displaystyle{\sum_{k+l+m=n}
\frac{n!}{k!l!m!}\mathcal{K}^{(\beta)}_{k,
\lambda}(x)\mathcal{K}^{(\beta)}_{l,
\lambda}(x)(e^{\beta}_{\lambda}\circ \xi_{q})^{(m)}(0)}$.
\item [\textbf{(P4)}]
$\mathcal{K}^{(\beta)}_{n,
\lambda}(x+y)=\displaystyle{\sum^{n}_{k=0}\binom{n}{k}\mathcal{K}^{(\beta)}_{k,
\lambda}(x)[y]_{n-k}}$,
\end{itemize}
where
\begin{equation}\label{def-of-[y]m}
[y]_{n}:=\displaystyle{\sum^{n}_{k=0}\binom{n}{k}q^{k}(y)_{1,
n-k}(-y)_{1, k}}.
\end{equation}
\end{thm}
\begin{proof}
\textbf{(P1)}
Firstly by using equation (\ref{Exp-alpha-1}), we get
$$\Psi_{\lambda}(x, z)=\displaystyle{\frac{e^{x\vartheta(z)}}{\mathcal{L}_{\mathcal{NB}^{(\lambda)}_{p, r}}(\vartheta(z))}
=\sum^{\infty}_{m=0}\frac{\vartheta(z)^{m}}{m!}A^{(\beta)}_{m,
\lambda}},$$ for any $z\in \mathcal{U}$. It is obvious that
$$\vartheta(z)=\log\Big(\frac{1+z}{1+qz}\Big)=\displaystyle{\sum^{\infty}_{k=0}\frac{z^{k}}{k!}\eta_{k}},$$
with
$$
\eta_{k}=\displaystyle{(-1)^{k+1}(k-1)!(1-q^{k}) \quad\forall
k>0,\quad \eta_{0}=0}.
$$
Therefore with the help of Eq. \eqref{def-lambda-mn} we have\\

$\displaystyle{\sum^{\infty}_{n=0}\frac{z^{n}}{n!}\mathcal{K}^{(\beta)}_{n,
\lambda}(x)}$
$$\begin{array}{lll}
\qquad&=&P^{(\beta)}_{0,
\lambda}+\displaystyle{\sum^{\infty}_{m=1}\frac{1}{m!}\Big(\sum^{\infty}_{k=0}
\frac{z^{k}}{k!}\eta_{k}\Big)^{m}P^{(\beta)}_{m, \lambda}(x)}\\
&=&P^{(\beta)}_{0,
\lambda}+\displaystyle{\sum^{\infty}_{m=1}\frac{1}{m!}\sum^{\infty}_{n=m=0}\frac{z^{n}}{n!}
\sum_{i_{1}+\cdots+i_{m}=n}\frac{n!}{i_{1}!\cdots i_{m}!}\prod^{m}_{j=1}\eta_{i_{j}}P^{(\beta)}_{m, \lambda}(x)}\\
&=&P^{(\beta)}_{0,
\lambda}+\displaystyle{\sum^{\infty}_{m=1}\frac{1}{m!}\sum^{\infty}_{n=m}\frac{z^{n}}{n!}
\sum_{i_{1}+\cdots+i_{m}=n}(-1)^{n+m}n!\prod^{m}_{j=1}\Big(\frac{1-q^{i_{j}}}{i_{j}}\Big)P^{(\beta)}_{m, \lambda}(x)}\\
&=&\displaystyle{\sum^{\infty}_{n=1}\frac{z^{n}}{n!}\sum^{n}_{m=0}\frac{\varpi(m,
n)}{m!}P^{(\beta)}_{m, \lambda}(x)}.
\end{array}$$
Hence we obtain
$$\displaystyle{\mathcal{K}^{(\beta)}_{n, \lambda}(x)=\displaystyle{\sum^{n}_{m=0}
\frac{\varpi(m, n)}{m!}P^{(\beta)}_{m, \lambda}(x)}}.$$
\begin{itemize}
\item [\textbf{(P2)}]
Let $\zeta$ be the inverse function of $\vartheta$. Then one can see
that
$$\zeta(z)=\displaystyle{\frac{e^{z}-1}{1-qe^{z}}=\sum^{\infty}_{k=0}\kappa_{k}\frac{z^{k}}{k!},}$$
where
$$\kappa_{k}=\displaystyle{\frac{d^{k}}{dz^{k}}\Big[\frac{e^{z}-1}{1-qe^{z}}\Big]_{z=0}}.$$
By using equation (\ref{wick-exp}), we get
$$\displaystyle{\frac{e^{xz}}{\mathcal{L}_{\mathcal{NB}^{(\lambda)}_{p, r}}(z)}=
\frac{e^{xz}}{e^{\beta}_{\lambda}\Big(r\log(\frac{p}{1-qe^{z}})\Big)}=
\sum^{\infty}_{m=0}\frac{\zeta(z)^{m}}{m!}\mathcal{K}^{(\beta)}_{m,
\lambda}(x)},\quad z\in \mathcal{U}.$$
So we have\\

$\displaystyle{\sum^{\infty}_{n=0}\frac{z^{n}}{n!}A^{(\beta)}_{n,
\lambda}(x)}=\mathcal{K}^{(\beta)}_{0, \lambda}(x)
+\displaystyle{\sum^{\infty}_{m=1}\frac{1}{m!}\Big(\sum^{\infty}_{k=1}\kappa_{k}\frac{z^{k}}{k!}\Big)^{m}\mathcal{K}^{(\beta)}_{m,
\lambda}(x)} $
$$\begin{array}{lll}
&=&\quad \mathcal{K}^{(\beta)}_{0,
\lambda}(x)+\displaystyle{\sum^{\infty}_{m=1}
\frac{1}{m!}\sum^{\infty}_{n=m}\frac{z^{n}}{n!}\sum_{i_{1}+\ldots+i_{m}=n}\frac{n!}{i_{1}!\ldots
i_{m}!}\prod^{m}_{j=1}\kappa_{i_{j}}
\mathcal{K}^{(\beta)}_{m, \lambda}(x)}\\
&=&\displaystyle{\sum^{\infty}_{n=0}\frac{z^{n}}{n!}
\sum^{n}_{m=0}\frac{\varrho(m, n)}{m!}\mathcal{K}^{(\beta)}_{m,
\lambda}(x)}.
\end{array}$$
Thus, we have
\begin{equation}\label{Aalphan}
P^{(\beta)}_{n,
\lambda}(x)=\displaystyle{\sum^{n}_{m=0}\frac{\varrho(m,
n)}{m!}\mathcal{K}^{(\beta)}_{m, \lambda}(x)}.
\end{equation}
On the other hand since
$$e^{xz}=\mathcal{N}\!\exp^{\beta}_{\lambda}(z, x)\mathcal{L}_{\mathcal{NB}^{(\lambda)}_{p, r}}(z),$$
we get
$$\begin{array}{lll}
\displaystyle{\sum^{\infty}_{n=0}\frac{x^{n}}{n!}z^{n}}&=&\displaystyle{\sum^{\infty}_{k=0}\frac{z^{k}}{k!}P^{(\beta)}_{k,
\lambda}(x)
\sum^{\infty}_{k=0}\frac{z^{k}}{k!}M_{\lambda}(k)}\\
&=&\displaystyle{\sum^{\infty}_{n=0}\frac{z^{n}}{n!}\Big(\sum^{n}_{k=0}\binom{n}{k}
P^{(\beta)}_{k, \lambda}(x)M_{\lambda}(n-k)\Big)}.
\end{array}$$
Hence with the help of Eq. \eqref{Aalphan}, we obtain
\begin{eqnarray}\label{eq-power-xn}
x^{n}&=&\displaystyle{\sum^{n}_{k=0}\binom{n}{k}
P^{(\beta)}_{k, \lambda}(x)M_{\lambda}(n-k)}\\
&=&\displaystyle{\sum^{n}_{k=0}\sum^{k}_{m=0}
\binom{n}{k}\frac{\varrho(m, k)}{m!}\mathcal{K}^{(\beta)}_{m,
\lambda}(x)M_{\lambda}(n-k)}.\nonumber
\end{eqnarray}
\item [\textbf{(P3)}]
By using the fact that
$$\mathcal{N}\!\exp^{\beta}_{\lambda}(\vartheta(z), x+y)=\mathcal{N}\!\exp^{\beta}_{\lambda}(\vartheta(z), x)
\mathcal{N}\!\exp^{\beta}_{\lambda}(\vartheta(z),
y)\mathcal{L}_{\mathcal{NB}^{(\lambda)}_{p, r}}(\vartheta(z)),$$ and
$$
\mathcal{L}_{\mathcal{NB}^{(\lambda)}_{p,
r}}(\vartheta(z))=e^{\beta}_{\lambda}\big(r\log(1+qz)\big)
=\displaystyle{\sum^{\infty}_{k=0}\frac{(e^{\beta}_{\lambda}\circ
\xi_{q})^{(k)}(0)}{k!}z^{k}},$$ we obtain
$$\begin{array}{lll}
\displaystyle{\sum^{\infty}_{n=0}\frac{z^{n}}{n!}\mathcal{K}^{(\beta)}_{n,
\lambda}(x+y)}
&=&\displaystyle{\sum^{\infty}_{k=0}\frac{z^{k}}{k!}\mathcal{K}^{(\beta)}_{k,
\lambda}(x)
\sum^{\infty}_{l=0}\frac{z^{l}}{l!}\mathcal{K}^{(\beta)}_{l,
\lambda}(y)
\sum^{\infty}_{m=0}\frac{z^{m}}{m!}\mathcal{K}^{(\beta)}_{m, \lambda}(x)}\\
&=&\displaystyle{\sum^{\infty}_{n=0}\frac{z^{n}}{n!}\sum_{k+l+m=n}
\frac{n!}{k!l!m!}\mathcal{K}^{(\beta)}_{k,
\lambda}(x)\mathcal{K}^{(\beta)}_{l,
\lambda}(x)(e^{\beta}_{\lambda}\circ \xi_{q})^{(m)}(0)}.
\end{array}$$
This gives the statement.
\item [\textbf{(P4)}]
Recall that, for $n\in \mathbb{N}$ and  $y\in \mathbb{C}$, the
generating function of the classical falling factorials
$(y)_{n}:=(y)_{1, n}$ obtained by \eqref{Degenerate-falling} with
$\lambda=1$ is
$$
\displaystyle{\sum^{\infty}_{n=0}\frac{z^{n}}{n!}(y)_{n}=\exp[y\log(1+z)]}.
$$
Then one can directly see that
\begin{eqnarray}\label{Taylor-exp-exp-v}
\exp(y\vartheta(z))&=&\exp(y\log(1+z)-y\log(1+qz))\nonumber\\
&=&\displaystyle{\sum^{\infty}_{n=0}[y]_{n}\frac{z^{n}}{n!}},
\end{eqnarray}
where
$$
[y]_{n}:=\displaystyle{\sum^{n}_{k=0}\binom{n}{k}q^{k}(y)_{n-k}(-y)_{k}}.
$$
It is clear that
$$\mathcal{N}\!\exp^{\beta}_{\lambda}(\vartheta(z), x+y)=\mathcal{N}\!\exp^{\beta}_{\lambda}(\vartheta(z), x)\exp(y\vartheta(z)).$$
Then with the help of equation \eqref{Taylor-exp-exp-v} we get
$$\begin{array}{lll}
\displaystyle{\sum^{\infty}_{n=0}\frac{z^{n}}{n!}\mathcal{K}^{(\beta)}_{n,
\lambda}(x+y)}&=&
\displaystyle{\sum^{\infty}_{k=0}\frac{z^{k}}{k!}\mathcal{K}^{(\beta)}_{k, \lambda}(x)\sum^{\infty}_{k=0}\frac{z^{k}}{k!}[y]_{k}}\\
&=&\displaystyle{\sum^{\infty}_{n=0}\frac{z^{n}}{n!}\sum_{k+m=n}\frac{n!}{k!m!}\mathcal{K}^{(\beta)}_{k, \lambda}(x)[y]_{m}}\\
&=&\displaystyle{\sum^{\infty}_{n=0}\frac{z^{n}}{n!}\sum^{n}_{k=0}
\binom{n}{k}\mathcal{K}^{(\beta)}_{k, \lambda}(x)[y]_{n-k}}.
\end{array}$$
Hence we deduce the expansion of $\mathcal{K}^{(\beta)}_{n,
\lambda}(x+y)$ given by
$$\mathcal{K}^{(\beta)}_{n, \lambda}(x+y)=\displaystyle{\sum^{n}_{k=0}
\binom{n}{k} \mathcal{K}^{(\beta)}_{k, \lambda}(x)[y]_{n-k}}.$$
\end{itemize}
\end{proof}
\begin{cor}
The polynomials $\mathcal{K}^{(\beta)}_{n, \lambda}$ has the
following form
\begin{equation}\label{Poly-Stir}
\mathcal{K}^{(\beta)}_{n,
\lambda}(x)=\displaystyle{\sum^{n}_{k=0}\sum^{k}_{j=0}(-1)^{k-j}\sum^{n-k-j}_{m=j}\binom{n}{m}s(m,
j)s(n-m, k-j)q^{n-m}P^{(\beta)}_{k, \lambda}(x)},
\end{equation}
where $s(n, k)$ is the Stirling numbers of the first kind defined
via its generating function by
\begin{equation}\label{Stir-1}
\displaystyle{\log(1+z)^{k}=k!\sum^{\infty}_{n=k}s(n,
k)\frac{z^{n}}{n!}}.
\end{equation}
\end{cor}
\begin{proof}
By applying the binomial series and by using \eqref{Stir-1} we get
$$\begin{array}{lll}
\vartheta(z)^{k}&=&(\log(1+z)-\log(1+qz))^{k}\\
&=&\displaystyle{k!\sum^{\infty}_{n=k}\frac{z^{n}}{n!}\sum^{k}_{j=0}(-1)^{k-j}\sum^{n-k-j}_{m=j}\binom{n}{m}s(m,
j)s(n-m, k-j)q^{n-m}}.
\end{array}$$
Then we obtain\\

$\displaystyle{\sum^{\infty}_{n=0}\frac{z^{n}}{n!}\mathcal{K}^{(\beta)}_{n,
\lambda}(x)}$
$$\begin{array}{lll}
\qquad&=&P^{(\beta)}_{0,
\lambda}(x)+\displaystyle{\sum^{\infty}_{k=1}\frac{1}{k!}(\vartheta(z))^{k}P^{(\beta)}_{m, \lambda}(x)}\\
&=&P^{(\beta)}_{0,
\lambda}(x)\\
&+&\displaystyle{\sum^{\infty}_{k=1}\sum^{\infty}_{n=k}\frac{z^{n}}{n!}
\sum^{k}_{j=0}(-1)^{k-j}\sum^{n-k-j}_{m=j}\binom{n}{m}s(m, j)s(n-m,
k-j)q^{n-m}P^{(\beta)}_{m, \lambda}(x)}\\
&=&\displaystyle{\sum^{\infty}_{n=0}\frac{z^{n}}{n!}\sum^{n}_{k=0}
\sum^{k}_{j=0}(-1)^{k-j}\sum^{n-k-j}_{m=j}\binom{n}{m}s(m, j)s(n-m,
k-j)q^{n-m}P^{(\beta)}_{m, \lambda}(x)}.
\end{array}$$
As a consequence we obtain the explicit expression of
$\mathcal{K}^{(\beta)}_{n, \lambda}(x)$ as in eq. \eqref{Poly-Stir}.
\end{proof}
\begin{cor}
The Appell polynomials $P^{(\beta)}_{n, \lambda}$ has the following
form in terms of $\mathcal{K}^{(\beta)}_{n, \lambda}(x)$:
\begin{equation}\label{Poly-Stir-2}
P^{(\beta)}_{n,
\lambda}(x)=\displaystyle{\sum^{n}_{k=0}\Big(\sum^{n}_{j=k}\binom{j-1}{k-1}\frac{q^{j-k}}{p^{j}}
j!S(n, j)\mathcal{K}^{(\beta)}_{k, \lambda}(x)}.
\end{equation}
\end{cor}
\begin{proof}
Let $u=e^{z}-1$, then we have
$$\begin{array}{lll}
\zeta^{k}(z)&=&\displaystyle{\frac{1}{(1-q)^{k}}u^{k}\Big(1-\frac{q}{1-q}u\Big)^{-k}}\\
&=&\displaystyle{\frac{1}{(1-q)^{k}}(e^{z}-1)^{k}\sum^{\infty}_{m=0}\binom{-k}{m}\Big(\frac{-q}{1-q}u\Big)^{m}}\\
&=&\displaystyle{\frac{1}{(1-q)^{k}}(e^{z}-1)^{k}\sum^{\infty}_{m=0}\binom{k+m-1}{m}\big(\frac{q}{p}\big)^{m}u^{m}}\\
&=&\displaystyle{\frac{1}{(1-q)^{k}}\sum^{\infty}_{m=0}\binom{k+m-1}{m}\big(\frac{q}{p}\big)^{m}(e^{z}-1)^{k+m}}.
\end{array}$$
Thus by using \eqref{Stir-2} we can write
$$\begin{array}{lll}
\zeta^{k}(z)&=&\displaystyle{\frac{1}{(1-q)^{k}}\sum^{\infty}_{m=0}\binom{k+m-1}{m}\big(\frac{q}{p}\big)^{m}
\sum^{\infty}_{n=k+m}(k+m)!S(n,
k+m)\frac{z^{n}}{n!}}\\
&=&\displaystyle{\frac{1}{p^{k}}\sum^{\infty}_{n=k}\frac{z^{n}}{n!}\sum^{n-k}_{m=0}\binom{k+m-1}{m}\big(\frac{q}{p}\big)^{m}
(k+m)!S(n,k+m)}\\
&=&\displaystyle{\frac{1}{p^{k}}\sum^{\infty}_{n=k}\frac{z^{n}}{n!}\sum^{n}_{j=k}\binom{j-1}{k-1}\big(\frac{q}{p}\big)^{j-k}
j!S(n,j)}\\
&=&\displaystyle{\sum^{\infty}_{n=k}\Big(\sum^{n}_{j=k}\binom{j-1}{k-1}\frac{q^{j-k}}{p^{j}}
j!S(n,j)\Big)\frac{z^{n}}{n!}}.
\end{array}$$
By using equation \eqref{wick-exp}, we have
$$\begin{array}{lll}
\displaystyle{\frac{e^{xz}}{\mathcal{L}_{\mathcal{NB}^{(\lambda)}_{p,
r}}(z)}}&=&\displaystyle{\sum^{\infty}_{n=0}\frac{z^{n}}{n!}P^{(\beta)}_{n,
\lambda}(x)}\\
&=&\displaystyle{\sum^{\infty}_{k=0}\frac{(\zeta(z))^{k}}{k!}\mathcal{K}^{(\beta)}_{k, \lambda}(x)}\\
&=&\displaystyle{\sum^{\infty}_{k=0}\Big[\sum^{\infty}_{n=k}\Big(\sum^{n}_{j=k}\binom{j-1}{k-1}\frac{q^{j-k}}{p^{j}}
j!S(n,j)\Big)\frac{z^{n}}{n!}\Big]\frac{\mathcal{K}^{(\beta)}_{k,
\lambda}(x)}{k!}}
\\
&=&\displaystyle{\sum^{\infty}_{n=0}\frac{z^{n}}{n!}\sum^{n}_{k=0}\Big(\sum^{n}_{j=k}\binom{j-1}{k-1}\frac{q^{j-k}}{p^{j}}
j!S(n,j)\Big)\frac{\mathcal{K}^{(\beta)}_{k, \lambda}(x)}{k!}}.
\end{array}$$
Hence we obtain the explicit expression of $P^{(\beta)}_{n,
\lambda}(x)$ in terms of $\mathcal{K}^{(\beta)}_{n, \lambda}(x)$
given as equation \eqref{Poly-Stir-2}.
\end{proof}
Now our goal to express the the degenerate Krawtchouk Appell
polynomials $\mathcal{K}^{(\beta)}_{n, \lambda}, n\in \mathbb{N}$,
in terms of Bell polynomials $B_{n,k}$ which are defined for any
$x_{1}, \ldots , x_{n-k+1}\in \mathbb{R}$ by
$$B_{n,k}((x_{i})^{n-k+1}_{j=1}):=
\displaystyle{\sum\frac{n!}{i_{1}!i_{2}!\ldots
i_{n-k+1}!}\Big(\frac{x_{1}}{1!}\Big)^{i_{1}} \ldots
\Big(\frac{x_{n-k+1}}{(n-k+1)!}\Big)^{i_{n-k+1}}},$$ where the
summation is over all sequences of non-negative integers $i_{1},
i_{2}, \ldots , i_{n-k+1}$ satisfying the following conditions:
\begin{itemize}
\item [(i)] $i_{1}+i_{2}+\ldots+i_{n-k+1}=k$,\\
\item [(ii)] $i_{1}+2i_{2}+\ldots+(n-k+1)i_{n-k+1}=n$.
\end{itemize}
Recall that the Fa\`{a} di Bruno's formula is given as follows:
\begin{equation}\label{Bruno-for}
\displaystyle{\frac{d^{n}}{dt^{n}}(\varphi(\psi(t)))=\sum^{n}_{k=0}\varphi^{(k)}(\psi(t))B_{n,
k}(\psi'(t), \psi''(t), \ldots, \psi^{(n-k+1)}(t))}.
\end{equation}
For more details see \cite{[j02]}.
\begin{lem}\label{lemma-tech1}
For every $n \in \mathbb{N}^{\ast}$, it holds that
$$P^{(\beta)}_{n, \lambda}(0)=\displaystyle{\sum^{n}_{k=0}(-1)^{k}k!B_{n, k}\Big(\big(M_{\lambda}(j)\big)^{n-k+1}_{j=1}\Big)}.$$
\end{lem}
\begin{proof}
First, applying \eqref{Exp-alpha-1}, we obtain
$$P^{(\beta)}_{n, \lambda}(0)=\displaystyle{\frac{d^{n}}{dz^{n}}\mathcal{N}\!\exp^{\beta}_{\lambda}(z, x)\big|_{z=x=0}}.$$
Let $\varphi_{1}$ and $\varphi_{2}$ be the two functions defined by
$$\begin{array}{lll}
\varphi_{1}(z)&:=&\mathcal{N}\!\exp^{\beta}_{\lambda}(z,
0)=\displaystyle{\frac{1}{\mathcal{L}_{\mathcal{NB}^{(\lambda)}_{p,
r}}(z)}=
\frac{1}{e^{\beta}_{\lambda}[r\log(\frac{p}{1-qe^{z}})]}}\\
\varphi_{2}(z)&:=&e^{xz},
\end{array}$$
then by the general Leibniz rule we get
$$\begin{array}{lll}
P^{(\beta)}_{n, \lambda}(0)&=&\displaystyle{\frac{d^{n}}{dz^{n}}\big(\varphi_{1}(z)\varphi_{2}(z)\big)|_{z=x=0}}\\
&=&\displaystyle{\sum^{n}_{k=0}
\binom{n}{k}\varphi^{(n-k)}_{1}(z)\varphi^{(k)}_{2}(z)|_{z=x=0}}.
\end{array}$$
It is clear that
$$\varphi^{(k)}_{2}(z)=x^{k}\varphi_{2}(z).$$
On the other hand, by using formula \eqref{Bruno-for} with the
choice
$$\varphi_{1}(z)=\sigma(\mathcal{L}_{\mathcal{NB}^{(\lambda)}_{p, r}}(z)),\quad \psi(z)=\frac{1}{z},$$
we get\\

$\varphi^{(n-k)}_{1}(z)$
$$\begin{array}{lll}
&=&\displaystyle{\sum^{n-k}_{j=0}\psi^{(j)}(\mathcal{L}_{\mathcal{NB}^{(\lambda)}_{p,
r}}(z))B_{n-k, j}(\mathcal{L}'_{\mathcal{NB}^{(\alpha)}_{p, r}}(z),
\mathcal{L}''_{\mathcal{NB}^{(\alpha)}_{p, r}}(z), \ldots,
\mathcal{L}^{(n-k-j+1)}_{\mathcal{NB}^{(\alpha)}_{p, r}}(z))}\\
&=&\displaystyle{\sum^{n-k}_{j=0}\frac{(-1)^{j}j!}{(\mathcal{L}_{\mathcal{NB}^{(\lambda)}_{p,
r}}(z))^{j+1}}B_{n-k,
j}\big((\mathcal{L}^{(i)}_{\mathcal{NB}^{(\lambda)}_{p,
r}}(z))^{n-k-j+1}_{i=1}\big)}.
\end{array}$$
Putting together yields\\

$\displaystyle{\frac{d^{n}}{dz^{n}}\big(\varphi_{1}(z)\varphi_{2}(z)\big)}$
$$\displaystyle{=
\sum^{n}_{k=0}\binom{n}{k}\Big[\sum^{n-k}_{j=0}\frac{(-1)^{j}j!}{(\mathcal{L}_{\mathcal{NB}^{(\lambda)}_{p,
r}}(z))^{j+1}} B_{n-k,
j}\big((\mathcal{L}^{(i)}_{\mathcal{NB}^{(\lambda)}_{p,
r}}(z))^{n-k-j+1}_{i=1}\big) \Big]x^{k}\varphi_{2}(z)}
$$
and the statement follows by evaluating the previous expression at
$z=x=0$.
\end{proof}
Combining \eqref{eq-power-xn} and Lemma \ref{lemma-tech1} with the
uniqueness of the representation in terms of the standard powers
$x^{k}$, we obtain the following technical result.
\begin{lem}
The Appell polynomials $P^{(\beta)}_{n, \lambda}$ has the following
form:
$$
P^{(\beta)}_{n, \lambda}(x)=\displaystyle{
\sum^{n}_{k=0}\binom{n}{k}\Big[\sum^{n-k}_{i=0}(-1)^{i}i! B_{n-k,
i}\big((M_{\lambda}(j))^{n-k-j+1}_{j=1}\big)\Big]x^{k}}.
$$
\end{lem}
Finally, we obtain the following main result, which gives an
explicit expression of the Krawtchouk Appell in terms of Bell
polynomials.
\begin{thm}
The degenerate Krawtchouk Appell polynomials
$\mathcal{K}^{(\lambda)}_{n, \beta}$ are given in terms of the Bell
polynomials as
\begin{equation}\label{Generalized-Appel-FP}
\mathcal{K}^{(\beta)}_{n, \lambda}(x)=\displaystyle{\sum^{n}_{k=0}
\binom{n}{k}\Big[\sum^{n-k}_{i=0}(-1)^{i}i!B_{n-k,
i}\Big((M_{\alpha}(j))^{n-k-i+1}_{j=1}\Big)\Big][x]_{k}}.
\end{equation}
\end{thm}
\begin{proof}
Firstly by using the general Leibniz rule we have
$$\begin{array}{lll}
\mathcal{K}^{(\lambda)}_{n, \beta}(0)&=&\displaystyle{\frac{d^{n}}{dz^{n}}\big(\varphi(z)\psi(z)\big)|_{z=x=0}}\\
&=&\displaystyle{\sum^{n}_{k=0}\binom{n}{k}\varphi^{(n-k)}(z)\psi^{(k)}(z)|_{z=x=0}},
\end{array}$$
with the choice
$$\begin{array}{lll}
\varphi(z)&:=&\mathcal{N}\!\exp^{\beta}_{\lambda}(z,
0)=\displaystyle{\frac{1}{\mathcal{L}_{\mathcal{NB}^{(\lambda)}_{p,
r}}(z)}=
\frac{1}{e^{\beta}_{\lambda}[r\log(\frac{p}{1-qe^{z}})]}}\\
\psi(z)&:=&\displaystyle{e^{x\vartheta(z)}=e^{x\log(\frac{1+z}{1+qz})}},
\quad |qe^{z}|<1.
\end{array}$$
Hence again by formula \eqref{Bruno-for} we obtain
$$
\varphi^{(n-k)}(z)=\displaystyle{\sum^{n-k}_{j=0}\frac{(-1)^{j}j!}{(\mathcal{L}_{\mathcal{NB}^{(\lambda)}_{p,
r}}(\vartheta(z)))^{j+1}} B_{n-k,
j}\big((\mathcal{L}^{i}_{\mathcal{NB}^{(\lambda)}_{p,
r}}(\vartheta(z)))^{n-k-j+1}_{i=1}\big)}
$$
and
$$
\psi^{(k)}(z)=\displaystyle{\sum^{k}_{l=0}x^{l}e^{x\vartheta(z)}
B_{k, l}\big((\vartheta^{(m)}(z))^{k-l+1}_{m=1}\big)}.
$$
Then we have
\begin{equation}\label{Appel-at0}
\mathcal{K}^{(\lambda)}_{n,
\beta}(0)=\displaystyle{\sum^{n}_{j=0}(-1)^{j}j!B_{n,
j}\Big((M_{\lambda}(m))^{n-j+1}_{m=1}\Big)}.
\end{equation}
Hence by using the property (P4) and (\ref{Appel-at0}) we get
$$\begin{array}{lll}
\mathcal{K}^{(\lambda)}_{n,
\beta}(x)&=&\displaystyle{\sum^{n}_{k=0}\binom{n}{k}
\mathcal{K}^{(\lambda)}_{n, \beta}(0)[x]_{k}}\\
&=&\displaystyle{\sum^{n}_{k=0}\binom{n}{k}
\Big[\sum^{n-k}_{i=0}(-1)^{i}i!B_{n-k,
i}\Big((M_{\lambda}(j))^{n-k-i+1}_{j=1}\Big)\Big][x]_{k}}.
\end{array}$$
\end{proof}

Now let $\mathcal{P}(\mathbb{R})$ the set of all polynomials in
$\mathbb{R}$ with complex coefficients defined by
$$\mathcal{P}(\mathbb{R})=\Big\{\varphi: \mathbb{R}\rightarrow \mathbb{C}; \varphi(x)=
\displaystyle{\sum^{N_{\varphi}}_{n=0}\varphi_{n}\mathcal{K}^{(\lambda)}_{n,
\beta}(x)}, \varphi_{n}\in \mathbb{C}, N_{\varphi}\in
\mathbb{N}^{\ast}\Big\}.$$ Let $z\in \mathbb{C}$ be given. Define
the scaling operator $\sigma_{z}$ as a continuous from
$\mathcal{P}(\mathbb{R})$ into itself such that
$$(\sigma_{z}\varphi)(x)=\varphi(zx), \quad \varphi\in \mathcal{P}(\mathbb{R}).$$
\begin{thm}
Let $\varphi\in \mathcal{P}(\mathbb{R})$ be given by its chaos
expansion
$$\varphi(x)=
\displaystyle{\sum^{N_{\varphi}}_{n=0}\varphi_{n}\mathcal{K}^{(\lambda)}_{n,
\beta}(x)}.$$ Then $\sigma_{z}\varphi$ is given explicitly as
follows:
\begin{equation}\label{scaling-expansion}
\sigma_{z}(\varphi)=\displaystyle{\sum^{N_{\varphi}}_{n=0}n!\varphi_{n}\sum^{n}_{k=0}\varepsilon_{n-k}(zx)\sum^{k}_{m=0}\frac{\rho(m,
k)(-\beta)_{m, \lambda}}{m!k!}},
\end{equation}
where $\varepsilon_{j}$ is given as equation \eqref{def-omega} and
$$\rho(m, k)=\displaystyle{\sum_{l_{1}+\ldots l_{m}=n}(-1)^{n+m}q^{m}k!\prod^{m}_{i=1}\frac{1}{l_{i}}}.$$
\end{thm}
\begin{proof}
Firstly one can see that
$$\begin{array}{lll}
e^{-\beta}_{\lambda}(\xi_{q}(t))&=&\displaystyle{\sum^{\infty}_{n=0}(-\beta)_{n,
\lambda}\frac{\xi^{n}_{q}(t)}{n!}}\\
&=&\displaystyle{1+\sum^{\infty}_{n=1}\frac{1}{n!}(-\beta)_{n,
\lambda}\Big(\sum^{\infty}_{l=0}(-1)^{l-1}rq^{l}\frac{t^{l}}{l}\Big)^{n}}\\
&=&\displaystyle{1+\sum^{\infty}_{n=1}\frac{1}{n!}(-\beta)_{n,
\lambda}\sum^{\infty}_{n=m=0}\frac{t^{n}}{n!}\sum_{l_{1}+\ldots+l_{m}=n}(-1)^{n+m}q^{m}\prod^{m}_{i=1}\frac{1}{l_{i}}}\\
&=&\displaystyle{\sum^{\infty}_{n=0}\frac{t^{n}}{n!}\sum^{n}_{m=0}\frac{\rho(m,
n)(-\beta)_{m, \lambda}}{m!}}.
\end{array}
$$
Hence we obtain
$$\begin{array}{lll}
\Psi_{\lambda}(zx,
t)&=&\displaystyle{\sum^{\infty}_{n=0}\frac{t^{n}}{n!}\mathcal{K}^{(\lambda)}_{n,
\beta}(zx)}\\
&=&\displaystyle{\frac{e^{xz\vartheta(t)}}{e^{\beta}_{\lambda}\{r\log(\frac{p}{1-qe^{\vartheta(t)}})\}}}\\
&=&e^{xz\vartheta(t)}e^{-\beta}_{\lambda}(r\log(1+qt))\\
&=&\displaystyle{\sum^{\infty}_{n=0}\varepsilon_{n}(zx)t^{n}\sum^{\infty}_{n=0}\sum^{n}_{m=0}\frac{\rho(m,
n)(-\beta)_{m, \lambda}}{m!n!}t^{n}}\\
&=&\displaystyle{\sum^{\infty}_{n=0}\sum^{n}_{k=0}\varepsilon_{n-k}(zx)\sum^{k}_{m=0}\frac{\rho(m,
k)(-\beta)_{m, \lambda}}{m!k!}t^{n}}.
\end{array}$$
Thus we get
$$\mathcal{K}^{(\lambda)}_{n,
\beta}(zx)=\displaystyle{n!\sum^{n}_{k=0}\varepsilon_{n-k}(zx)\sum^{k}_{m=0}\frac{\rho(m,
k)(-\beta)_{m, \lambda}}{m!k!}}$$ and as a consequence we obtain
$$\begin{array}{lll}
\sigma_{z}(\varphi)&=&\displaystyle{\sum^{N_{\varphi}}_{n=0}n!\varphi_{n}\mathcal{K}^{(\lambda)}_{n,
\beta}(zx)}\\
&=&\displaystyle{\sum^{N_{\varphi}}_{n=0}n!\varphi_{n}\sum^{n}_{k=0}\varepsilon_{n-k}(zx)\sum^{k}_{m=0}\frac{\rho(m,
k)(-\beta)_{m, \lambda}}{m!k!}}.
\end{array}$$
\end{proof}
\begin{ex}
If $\varphi(x)=
\displaystyle{\sum^{N_{\varphi}}_{n=0}\varphi_{n}\mathcal{K}^{(\lambda)}_{n,
\beta}(x)}$, then by using $\textbf{(P4)}$, the translation operator
$\tau_{y}$ for $y\in \mathbb{R}$ can be represented as
$$\begin{array}{lll}
\tau_{y}\varphi(x)&=&\varphi(x+y)\\
&=&\displaystyle{\sum^{N_{\varphi}}_{n=0}\varphi_{n}\mathcal{K}^{(\lambda)}_{n,
\beta}(x+y)}\\
&=&\displaystyle{\sum^{N_{\varphi}}_{n=0}\sum^{n}_{k=0}\binom{n}{k}\varphi_{n}\mathcal{K}^{(\lambda)}_{n,
\beta}(x)[y]_{n-k}}.
\end{array}$$
Furthermore, the action of the translation operator on the
normalized-Laplace exponential function is given by
$$\tau_{y}\mathcal{N}\exp^{\beta}_{\lambda}(z, x)=e^{yz}\mathcal{N}\exp^{\beta}_{\lambda}(z, x).$$
Let us define an operator $D_{y}$ on $\mathcal{P}(\mathbb{R})$ such
that it action on $\mathcal{N}\exp^{\beta}_{\lambda}(z, x)$ is given
by
$$D_{y}\mathcal{N}\exp^{\beta}_{\lambda}(z, x)=yz\mathcal{N}\exp^{\beta}_{\lambda}(z, x).$$
Then we obtain
$$\begin{array}{lll}
\big\langle\tau_{y}\mathcal{N}\exp^{\beta}_{\lambda}(z, x),
\mathcal{N}\exp^{\beta}_{\lambda}(z, x')\big\rangle&=&\big\langle
e^{yz}\mathcal{N}\exp^{\beta}_{\lambda}(z, x),
\mathcal{N}\exp^{\beta}_{\lambda}(z, x')\big\rangle\\
&=&e^{yz}\big\langle\mathcal{N}\exp^{\beta}_{\lambda}(z, x),
\mathcal{N}\exp^{\beta}_{\lambda}(z, x')\big\rangle.
\end{array}$$
On the other hand, we have\\

$\big\langle\exp(D_{y})\mathcal{N}\exp^{\beta}_{\lambda}(z, x),
\mathcal{N}\exp^{\beta}_{\lambda}(z, x')\big\rangle$
$$\begin{array}{lll}
&=&\displaystyle{\Big\langle
\sum^{\infty}_{k=0}\frac{1}{k!}D^{k}_{y}\mathcal{N}\exp^{\beta}_{\lambda}(z,
x),
\mathcal{N}\exp^{\beta}_{\lambda}(z, x')\Big\rangle}\\
&=&\displaystyle{\Big\langle
\sum^{\infty}_{k=0}\frac{1}{k!}(yz)^{k}\mathcal{N}\exp^{\beta}_{\lambda}(z,
x),
\mathcal{N}\exp^{\beta}_{\lambda}(z, x')\Big\rangle}\\
&=&e^{yz}\langle\mathcal{N}\exp^{\beta}_{\lambda}(z, x),
\mathcal{N}\exp^{\beta}_{\lambda}(z, x')\rangle.
\end{array}$$
Hence we conclude that $\tau_{y}=\exp(D_{y})$.
\end{ex}

%=====================================================================================================================================
\section*{Conclusion}
In this study, we propose a concrete method for investigating
several combinatorial properties and highlight specific aspects of a
new class of polynomial sequences known as degenerate
Krawtchouk-Appell polynomials. Our research focuses primarily on the
role of the relationship between these new polynomials, Stirling
numbers, the scaling operator, translation operator and orthogonal
Bell polynomials. This combinatorial study will serve us in our
future research to develop some stochastic applications, with the
aim of obtaining practical applications.

\section*{Acknowledgments}
The authors extend their appreciation to Ongoing Research Funding Program, (ORF-2026-1068), King Saud University, Riyadh, Saudi Arabia.
\section*{Author Contributions}
All authors reviewed the manuscript.
\section*{Data Availability}
No data sets were generated or analyzed during the current study.
\section*{Declarations}
Competing interests: The authors declare no competing interests.
%============================================================================================================================================

% -----------------------------------------------------------------------
\end{document}